\documentclass[12pt,reqno]{amsart}

\usepackage{amsmath, amssymb}
\usepackage[frame,cmtip,arrow,matrix,line,graph,curve]{xy}
\usepackage{graphpap,color,paralist,pstricks}
\usepackage[mathscr]{eucal}
\usepackage[pdftex]{graphicx}
\usepackage[pdftex,colorlinks,backref=page,citecolor=blue]{hyperref}
\usepackage{tikz}
\usepackage{kotex}
\usepackage{array}
\usepackage{pgfplots}
\usepackage{mathrsfs}
\usetikzlibrary{arrows} 
\usepackage{tikz-cd}
\usepackage{mathtools}
\numberwithin{equation}{section}

\newcolumntype{P}[1]{>{\centering\arraybackslash}p{#1}}
\newcolumntype{M}[1]{>{\centering\arraybackslash}m{#1}}

\newtheorem{theorem}{Theorem}[section]
\newtheorem{proposition}[theorem]{Proposition}
\newtheorem{corollary}[theorem]{Corollary}
\newtheorem{lemma}[theorem]{Lemma}

\theoremstyle{definition}
\newtheorem{definition}[theorem]{Definition}

\newtheorem{remark}[theorem]{Remark}

\newcommand{\PP}{\mathbb{P}}

\newcommand{\CC}{\mathbb{C}}

\newcommand{\ZZ}{\mathbb{Z}}

\newcommand{\cO}{\mathcal{O} }

\newcommand{\cF}{\mathcal{F} }

\newcommand{\cH}{\mathcal{H} }

\newcommand{\cK}{\mathcal{K} }

\newcommand{\cQ}{\mathcal{Q} }

\newcommand{\cU}{\mathcal{U} }

\newcommand{\rG}{\mathrm{G} }
\newcommand{\rH}{\mathrm{H} }

\newcommand\bH{\mathbf{H}}

\newcommand\bM{\mathbf{M}}

\newcommand\bS{\mathbf{S}}

\def\Sym{\mathrm{Sym} }
\def\Hom{\mathrm{Hom} }
\def\Ext{\mathrm{Ext} }

\def\Gr{\mathrm{Gr} }

\def\lr{\rightarrow}

\newcommand{\rank}{\mathrm{rank}\, }

\newcommand{\ses}[3]{0\lr{#1}\lr{#2}\lr{#3}\lr 0}

\hypersetup{%
pdftitle={Hilbert scheme of conics in del Pezzo varieties},%
pdfauthor={Kiryong Chung, Bomyeong Kim},%
pdfkeywords={Del Pezzo variety, Conic, Torus fixed curve, Deformation-obstruction space},%
citecolor=blue,%
linkcolor=blue,%
}

\begin{document}

\title[Local structure of conics in quintic del Pezzo varieties]{Local structure of the Hilbert scheme of conics in quintic del Pezzo varieties}
\date{December 26, 2025}

\author{Kiryong Chung}
\address{Department of Mathematics Education, Kyungpook National University, 80 Daehakro, Bukgu, Daegu 41566, Korea}
\email{krchung@knu.ac.kr}

\author{Bomyeong Kim}
\address{Department of Mathematical Sciences, KAIST, 291 Daehak-ro, Yuseong-gu, Daejeon 34141, Republic of Korea}
\email{kkbm2000@kaist.ac.kr}

\author{Minseong Kwon}
\address{Morningside Center of Mathematics, Academy of Mathematics and Systems Science, Chinese Academy of Sciences, Beijing 100190, China}
\email{minseong@amss.ac.cn}

\keywords{Quintic del Pezzo variety, Conic, Torus fixed curve, Deformation-Obstruction space}
\subjclass[2020]{14C05, 14H10, 14J45}

\begin{abstract}
Let $X$ be the quintic del Pezzo $4$-fold. It is very well-known that $X$ is realized by a smooth linear section of Grassmannian $\mathrm{Gr}(2,5)$. In this paper, we prove that the Hilbert scheme of conics in $X$ is a smooth variety of dimension $7$ by using a torus action on $X$, which provides a more direct proof about the first named author's previous result.
\end{abstract}

\maketitle
\section{Introduction and results}
By definition, the quintic del Pezzo $4$-fold $X$ is a smooth projective variety such that the anti-canonical line bundle of $X$ is isomorphic to $-K_{X}\cong -3L$ for some ample generator $L\in \text{Pic}(X)\cong \ZZ$ with $L^{4}=5$. By the work of Fujita (\cite{Fuj81}), it is known that $X$ is isomorphic to a linear section of Grassmannian variety $\Gr(2,V_5)$, $\dim V_5=5$. For the presentation of the Pl\"{u}cker coordinates of $X$ with a torus action $T=\CC^*$, let us denote the vector space $V_5=\mathrm{span}_{\CC} \{e_0, e_1, e_2, e_3, e_4\}$ with weights $2$, $0$, $1$, $-1$, $-2$ respectively. Under the Pl\"{u}cker embedding of $\Gr(2, V_5)$, let us denote by $p_{ij}$ the Pl\"{u}cker coordinates of $\PP(\wedge^2V_5)$ where $p_{ij}(e_k\wedge e_l)=\delta_{ik}\delta_{jl}$ for $i <j$, $k<l$. From now on, for the fixed two hyperplanes
\begin{equation}\label{defhy}
H_1: p_{12}-p_{03}=0,\; H_{-1}: p_{13}-p_{24}=0.
\end{equation}
let us define the quintic del Pezzo fourfold $X$ as the intersection $X=\Gr(2,V_5)\cap H_1\cap H_{-1}$. Note that torus weights of the hyperplanes $H_i$ are given by $i\in\{\pm1\}$. For a fixed embedding of a smooth projective variety $Y\subset \PP^r$, let $\bH_d(Y)$ be the Hilbert scheme parameterizing curves $C$ in $Y$ with Hilbert polynomial $\chi(\cO_C(m))=dm+1$.  In this paper, by studying the torus fixed conics in the quintic del Pezzo variety $X$, we will show that $\bH_2(X)$ is smooth. In \cite{CHL18}, the authors proved this fact by an indirect method. That is, under the fact that $\bH_2(\Gr(2, 5))$ is smooth (\cite[Theorem 3.1]{IM09}), they checked that $\bH_2(X)$ is a reduced scheme by a general theory (i.e., generically reduced $+$ locally complete intersection $+$ Cohen-Macaulay space $\Rightarrow$ reduced). In this paper, the most technical part of proof is that there exist, so called, non-free lines in $X$ (Definition \ref{deffree}). To deal with this case, we need to consider the restricted normal bundle of singular conics. Related with this, in \cite{BF86}, the authors studied threefold case. As a generalization of this one, in the paper \cite[Appendix 2]{KPS18}, the author present a general form for any dimension. In this paper, by applying these result to $\CC^*$-fixed conics and computing the deformation spaces, we obtain
\begin{theorem}[Section 3.4]\label{mainthm}
The Hilbert scheme $\bH_2(X)$ is a smooth variety. Furthermore, the space $\bH_2(X)$ is a rational variety.
\end{theorem}
On the other hand, in the excellent paper \cite{IM09}, the author reconstructed a compact hyper-K\"{a}hler manifold of dimension $4$ by using the Hilbert scheme of conics in a Fano $4$-fold, where it was defined by a hyperplane and hyperquadric section of $\Gr(2, 5)$. This paper may play a role to well understand the stream and the proof of the result \cite{IM09} related with geometry of conics space. As final comment, by this paper, we confirm again that the Hilbert scheme of conics in any quintic del Pezzo variety is always smooth whenever it exists (\cite[Lemma 3.2]{CHL18}, \cite[Proposition 1.2.2]{Ili94}, Theorem \ref{mainthm} and Proposition \ref{consm}).
\subsection*{Acknowledgements}
The first named author gratefully acknowledges encouraging comments of Jun-Muk Hwang for the deformation space of a double line supported on a \emph{non-free lines} and the question (through an email to the first named author) of O. Debarre related to an explicit citation of the smoothness of Hilbert scheme.
The authors would like to thank Lorenzo Barban for valuable comments on the draft of this paper.
The third named author was partially supported by the National Research Foundation of Korea (NRF) grant funded by the Korea government (MSIT) (RS-2021-NR062093).

\section{Preliminary}
For later use, let us collect well-known facts.
\subsection{General results related with a torus action}\label{bbthm}
The following lemma is an immediate consequence of the Borel fixed point theorem:
\begin{lemma}\label{fsws}
Let $\bM$ be a closed subscheme of a projective space $\PP^r$. Assume that $\bM$ is equipped with a $\CC^*$-action. If $\bM$ is smooth at each point in $\bM^{\CC^*}$, then $\bM$ is a smooth variety.
\end{lemma}

The following easy lemma will be used several times, and so we put the statement.
\begin{lemma}[\protect{\cite[Lemma~3.1]{CKK25}}]\label{flocp}
Let $V$ be the finite-dimensional vector space with the diagonal $\CC^*$-action of distinct weights. If the repetition of weights of the coordinate points in $\PP(\wedge^2V)$ is less than or equal to two, the fixed loci of the Grassmannian variety $\Gr(2, V)$ are isolated.
\end{lemma}

Let $Y$ be a smooth projective variety with a $\CC^*$-action. Then the $\CC^*$-invariant loci of $Y$ are decomposed into connected components $Y^{\CC^*}=\bigsqcup\limits_{i} Y_i$. For each $Y_i$, the $\CC^*$-action on the tangent bundle $TY|_{Y_i}$ provides a decomposition as $TY|_{Y_i}=T^+\oplus T^{0}\oplus T^{-}$ where $T^+$, $T^{0}$ and $T^{-}$ are the subbundles of $TY|_{Y_i}$ such that the group $\CC^*$ acts with positive, zero and negative weights respectively. Under the local linearization, $T^0\cong TY_i$ and
\[
T^+\oplus T^{-}=N_{Y_i/Y}=N^{+}(Y_i)\oplus N^{-}(Y_i)
\]
is the decomposition of the normal bundle of $Y_i$ in $Y$. Let  $\mu^{\pm}(Y_i)=\rank N^{\pm}(Y_i)$. A fundamental result in the theory of $\CC^*$-action on $Y$ has been provided by A. Bia{\l}ynicki-Birula. Let
\[
Y^{+}(Y_i)=\{y\in Y \mid \lim_{t\to 0} t\cdot y \in Y_i\} \; \text{and}\; Y^{-}(Y_i)=\{y\in Y \mid \lim_{t\to \infty} t\cdot y \in Y_i\}.
\]
\begin{theorem}[\protect{\cite{Bir73}}]\label{prop:onetoone}
Under the above assumptions and notations,
\begin{enumerate}
\item $Y=\bigsqcup\limits_{i} Y^{\pm}(Y_i)$ and
\item For each component $Y_i$, the map $Y^{+}(Y_i) \rightarrow Y_{i}$, $y \mapsto \lim_{t \rightarrow 0} t \cdot y$ (resp. $Y^{-}(Y_i) \rightarrow Y_{i}$, $y \mapsto \lim_{t \rightarrow \infty} t \cdot y$) is a Zariski locally trivial fibration whose fiber is the affine space $\CC^{\mu^{+}(Y_i)}$ (resp. $\CC^{\mu^{-}(Y_i)}$).
\end{enumerate}
In particular, if $\mu^{\pm}(Y_i)=\dim Y$ for some $i$ and $Y$ is connected, then $Y$ is a rational variety.
\end{theorem}
\subsection{Lines and planes in $X$}
It is well-known that the Hilbert scheme of lines in $\Gr(2, V_5)$ is isomorphic to \[\bH_1(\Gr(2, V_5))\cong \Gr(1, 3, V_5)\] by \cite[Exercise 6.9]{Har95}. A line in $\Gr(2, V_5)$ is determined by a pencil of lines passing through a point (so called, \emph{vertex}). Also, under the natural embedding $\bH_1(X)\subset \bH_1(\Gr(2, V_5))$, we have a $\CC^*$-equivariant projection map
\begin{equation}\label{blmap}
p:\bH_1(X)\subset\bH_1(\Gr(2, V_5))\cong \Gr(1, 3, V_5)\lr \Gr(1, V_5)
\end{equation}
which associates to a line its \emph{vertex}.
\begin{proposition}[\protect{\cite[Proposition 2.7]{Pro94} and \cite[Lemma 6.3]{CHL18}}]\label{linespace}
The composition map
\begin{equation}\label{blqd4}
p: \bH_1(X)\to \Gr(1,V_5)
\end{equation}
in \eqref{blmap} is a smooth blow-up along a smooth conic $C_{v}\subset \Gr(1,V_5)$. Here, the equation of the conic $C_{v}$ (called the \emph{vertex conic}) is given by
\begin{equation}\label{vconic}
C_{v}=\{[a_0:a_1:a_2:a_3:a_4]\mid a_0a_4+a_1^2=a_2=a_3=0\}\subset \Gr(1,V_5).
\end{equation}
\end{proposition}
The normal bundle $N_{L/X}$ of a line $L$ in $X$ is one of the following types (\cite[Lemma 1.6]{PZ16}):
\begin{equation}\label{norinx}
N_{L/X}\cong \cO_L(1)\oplus \cO_L^{\oplus 2}\; \mathrm{or} \; \cO_L(-1)\oplus \cO_L(1)^{\oplus 2}.
\end{equation}
\begin{definition}\label{deffree}
In \eqref{norinx}, let us call a line of the first type (resp. the second type) \emph{free line} (resp. \emph{non-free line}).
\end{definition}
In \cite[Corollary 3.2]{Chu23}, the author proved that the locus of non-free lines is isomorphic to a $\PP^1$-fibration over the vertex conic $C_{v}(\cong \PP^1)$, which is a sublocus of the exceptional divisor of the blow-up map $p$ in \eqref{blqd4}.
\begin{proposition}\label{nofreel}
There are exactly four $\CC^*$-fixed, non-free lines in $X$ (Table \ref{table_1}).
\end{proposition}
\begin{proof}
By Lemma \ref{flocp}, the fixed locus of $\Gr(1, V_5)$ is just the coordinate points since the weights are different to each other. For each point $\PP(e_i)$, $0\leq i \leq 4$, the fixed point of the fiber $p^{-1}(\PP(e_i))$ in \eqref{blmap} determines the fixed loci of $\bH_1(X)$. But the fiber of the map $p$ is isomorphic to $\Gr(2, V_5/\langle e_i\rangle)$. Hence by Lemma \ref{flocp} again, the $\CC^*$-fixed lines in $\Gr(2, V_5)$ are of the form
\[
\{e_i\wedge(ue_s+ ve_t)\}, [u:v]\in \PP^1
\]
for $s, t\in \{0,1,2,3,4\}\setminus \{i\}$. Plugging above lines in the defining relation \eqref{defhy} of del Pezzo $4$-fold $X$ and then using the result in proof of \cite[Corollary 3.2]{Chu23}, we have non-free lines as listed in Table \ref{table_1}.
Note that in Table \ref{table_1}, (a) (resp. (c)) is symmetric to (b) (resp. (d)) with respect to the permutation $(04)(23)$.
\end{proof}

\begin{table}[h]
   \setlength{\tabcolsep}{4pt}
   \centering
   \begin{tabular}{|c||c|c|} 
    \hline
    \rule{0pt}{2.6ex} 
    \emph{Name} & Non-free line &   \\ [0.5ex] 
    \hline
    &  &   \\ [-1.5ex]
    \emph{(a)} & $\{e_0\wedge(ue_1+ ve_2)\}$ & $L_{1}$ \\[1.5ex]
    \hline
    \emph{(b)} & $\{e_4\wedge(ue_1+ ve_3)\}$ &  $L_{0}$ \\[1.5ex]
    \hline
    \emph{(c)} & $\{e_0\wedge(ue_1+ ve_4)\}$ &  $L_{-1}$ \\[1.5ex]
    \hline
    \emph{(d)} &$\{e_4\wedge(ue_0+ ve_1)\}$ &  $L_{-2}$ \\[1.5ex]
    \hline
 \end{tabular}  
   \caption{$\mathbb{C}^{\ast}$-fixed non-free lines parameterized by $[u:v]\in \PP^1$}
   \label{table_1}
 \end{table}
 
\begin{remark}
There exist five $\CC^*$-fixed, free lines in $X$ as follows: $\{e_0\wedge(ue_2+ve_4)\}$, $\{e_4\wedge(ue_0+ve_3)\}$, $\{e_1\wedge(ue_0+ve_4)\}$, $\{e_2\wedge(ue_0+ve_3)\}$, $\{e_3\wedge(ue_2+ve_4)\}$ for $[u:v]\in \PP^1$.
\end{remark}

The moduli space of planes in $X$ and the positional relation of planes were described in \cite[(10.2)]{Fuj81} and \cite[Proposition 2.2]{Pro94}. Note that the homology group $\rH_4(X, \ZZ)$ is freely generated by the Schubert classes $\sigma_{3,1}$ and $\sigma_{2,2}$. Under the choice of two hyperplanes in \eqref{defhy}, we have 
\begin{proposition} [\protect{\cite[Lemma 6.3]{CHL18}}]\label{planespace}
Each point $t\in C_{v}(\cong\PP^1)$ parameterizes the $\sigma_{3,1}$-type planes $P_t$ such that the vertex of the plane $P_t$ is the point $\{t\}$ in the vertex conic $C_{v}$ in \eqref{vconic}. In fact, the planes $P_t$ in $X$ are $P_t=\PP(V_1^t\wedge V_4^t)$ where $V_1^t=\text{span}\{e_0+te_1-t^2e_4\}$ and $V_4^t=\text{span}\{e_0,e_1,e_2+te_3,e_4\}$. Also the unique $\sigma_{2,2}$-type plane $S$ is given by $S=\PP(\wedge^2 \text{span}\{e_0,e_1,e_4\})$.
\end{proposition}

\subsection{Conics in $X$ via a Grassmannian bundle}
Let $\cU$ be the universal subbundle over the Grassmannian $\mathrm{Gr}(4,V)$. Set $\rG=\mathrm{Gr}(2,V)$.
Let $\bS(\mathrm{G})=\Gr(3,\wedge^2\cU)$ be the Grassmannian bundle over $\mathrm{Gr}(4,V)$. The space $\bS(\rG)$ is an incidence variety of pairs 
\[\bS(\rG)=\{(U,V_4)\,|\,U\subset \wedge^2 V_4\}\subset \mathrm{Gr}(3,\wedge^2 V)\times \mathrm{Gr}(4, V).\]
The correspondence
\[
(U,V_4)\mapsto \PP(U)\cap \Gr(2,V_4)
\]
between $\bS(\rG)$ and $\bH_2(\rG)$ provides an birational map 
\begin{equation}\label{birG}
\Phi: \bS(\rG)\dashrightarrow \bH_2(\rG).
\end{equation}
Let us denote by $T(\rG)$ the undefined locus of the map $\Phi$. Then the space $T(\rG)$ is isomorphic to the disjoint union of flag varieties
\begin{equation}\label{undefG}
\Gr(1,4,5)\sqcup \Gr(3,4,5)\cong T(\rG)\subset \bS(\rG),
\end{equation}
where the embedding map is defined as follows:
\begin{itemize}
\item a pair $(V_1,V_4)\in \Gr(1,4,5)$ is sent to $(V_1\wedge V_4, V_4)\in T(\rG)$ and 
\item a pair $(V_3,V_4)\in \Gr(3,4,5)$ is sent to $ (\wedge^2V_3, V_4)\in T(\rG)$.
\end{itemize}
Abstractly, the locus $T(G)$ is related with the Fano scheme of planes in $\rG$. By a restriction of the map $\Phi$ in \eqref{birG}, we have a birational model of Hilbert scheme $\bH_2(X)$ as follows. 
Let
\[\cK= \mathrm{ker}\{\wedge^2\cU\subset \wedge^2 \cO^{\oplus 5}\to \cO^{\oplus 2} \}\] be the kernel of the composition map where the arrow is given by $\{p_{12}-p_{03},p_{13}-p_{24}\}$. The sheaf $\cK$ is locally free over $\Gr(4,5)$ with $\text{rank} \,\cK=4$. Let $\bS(X)=\Gr(3, \cK)$ be the relative Grassmannian over $\Gr(4,5)$. Let $T(X)= \bS(X)\cap T(\rG)$ be the (scheme theoretic) intersection of $\bS(X)$ and $T(\rG)$ in $\bS(\rG)$. The space $\bS(X)$ is the incident variety
\begin{equation}\label{incident}
\bS(X)=\{(U_3, V_4) \mid U_3\subset \cK_{[V_4]}\}\subset \bS(\rG).
\end{equation}
By restricting the map $\Phi$ in \eqref{birG} to the space $\bS(X)$, we have a birational correspondence 
\begin{equation}\label{resmap1}
\Phi:\bS(X)\dashrightarrow  \bH_2(X), \;\Phi([(U_3, V_4)])=\PP(U_3)\cap \Gr(2,V_4).
\end{equation}
Here we use the same notation of the restriction map as the original map $\Phi$. Note that the inverse map $\Phi^{-1}$ corresponds a conic $C$ to the pair of its linear spanning $\langle C\rangle=\PP^2$ and the linear spanning of the rational normal scroll of type $(1,1)$ or $(0, 2)$.

On the other hand, the undefined locus $T(X)$ of the map $\Phi$ in \eqref{resmap1} is isomorphic to a disjoint union $\PP^1\sqcup \PP^1$, which parameterizes two distinct points $[P_t]$ and $[S]$ over a linear subspace $(t\in) \PP^1\subset\Gr(4,5)$ (\cite[Remark 6.8]{CHL18}). By blowing-up the space $\bS(X)$ along the undefined locus $T(X)$, one can extend the rational map $\Phi$ to a regular morphism.
\begin{proposition}[\protect{\cite[Proposition 6.7 and Remark 6.8]{CHL18}}]\label{bidiagconic}
Under above definition and notation, The Hilbert scheme $\bH_2(X)$ is obtained from $\bS(X)$ by birational maps:
\begin{equation*}
\xymatrix{
\widetilde{\bS}(X)\ar[d]\ar[rd]^{\widetilde{\Phi}}&\\
\Phi:\bS(X)\ar@{-->}[r]& \bH_2(X),
}
\end{equation*}
where the space $\widetilde{\bS}(X)$ is a relative conics space over $\Gr(4,5)$ such the fiber over $\Gr(4,5)$ is the Hilbert scheme $\bH_2(\Gr(2,4)\cap H_1\cap H_{-1})$ of conics in a quadric surface $\Gr(2,4)\cap H_1\cap H_2$.
\end{proposition}
Let $\Gr(2, \cU)$ be the relative Grassimannian over $\mathrm{Gr}(4, V)$. From the universal  sequence $\ses{\cU}{\cO\otimes V_5}{\cQ}$, we have a natural projection map $\Gr(2, \cU)\lr \Gr(2, V_5)$. Hence there is the structure morphism $\pi: \bH_2(\Gr(2, \cU))\lr \bH_2(\Gr(2, V))$ where $\bH_2(\Gr(2, \cU))$ is the relative Hilbert scheme of conics. The extended morphism $\widetilde{\Phi}$ in Proposition \ref{bidiagconic} is nothing but the restriction of the structure map $\pi$ along the embedding $\bH_2(X)\subset \bH_2(\Gr(2, V))$.

A planar double line $D$ in $X$ supported on $L$ is classified by $\Ext_X^1(\cO_L, \cO_L(-1))\cong \rH^0(N_{L/X}(-1))$ up to scalar multiplication. The double line $D$ for a free (resp. non-free) line $L$ in $X$ is unique (resp. corresponds to a point in $\PP^1$) by the normal bundle type of line in $X$.

\begin{proposition}
There exist exactly two $\CC^*$-fixed double lines (Table \ref{table_2}) supported on each non-free line in Table \ref{table_1}.
\end{proposition}

\begin{proof}
Since the proof is similar to each other, we only prove the claim for the case (a) in Table \ref{table_1}. By the proof of \cite[Corollary 3.2]{Chu23}, the non-free line is supported on one parameter family of quadric cones defined by
\[
\langle up_{04}^2+(up_{01}- p_{02})p_{14}, p_{12}-p_{03}, p_{12}-up_{04}, p_{13}-p_{24}, p_{23}-up_{24}, p_{24}-up_{14}, p_{34}\rangle
\]
for $u\in \CC^*$ in $\PP(\wedge^2V)=\PP^9$. By computing one parameter family of tangent planes of the quadric cone along the line, the $\CC^*$-fixed double lines are given as in Table \ref{table_2}.
\end{proof}

\begin{table}[h]
   \setlength{\tabcolsep}{4pt}
   \centering
   \begin{tabular}{|c||c|c|} 
    \hline
    \rule{0pt}{2.6ex} 
    \emph{Name} & Ideal $I_{D/\PP^9}$ & $D_{\text{red}}$  \\ [0.5ex] 
    \hline
    &  &   \\ [-1.5ex]
    \emph{(A-1) or (A-2)} & $\langle p_{04}^2, p_{03}, p_{12}, p_{13}, p_{14}, p_{23}, p_{24}, p_{34}\rangle$ or $\langle p_{12}^2, p_{04}, p_{12}-p_{03}, p_{13}, p_{14}, p_{23}, p_{24}, p_{34}\rangle$ & (a) \\[1.5ex]
    \hline
    \emph{(B-1) or (B-2)} & Permutation $(04)(23)$ of (A-1) and (A-2) respectively &  (b) \\[1.5ex]
    \hline
    \emph{(C-1) or (C-2)} & $\langle p_{02}^2, p_{03}, p_{12}, p_{13}, p_{14}, p_{23}, p_{24}, p_{34}\rangle$ or $\langle p_{14}^2, p_{02}, p_{03}, p_{12}, p_{13}, p_{23}, p_{24}, p_{34}\rangle$ &  (c) \\[1.5ex]
    \hline
    \emph{(D-1) or (D-2)} &Permutation $(04)(23)$ of (C-1) and (C-2) respectively&  (d) \\[1.5ex]
    \hline
 \end{tabular} 
   \caption{$\mathbb{C}^{\ast}$-fixed double lines supported on a non-free line}
   \label{table_2}
 \end{table} 
\subsection{Normal bundles of broken conics}
Let us recall the result of \cite[Section 1]{BF86} and \cite[Appendix 2]{KPS18} related with the normal bundle of a broken conic.
Let $C=L_1\cup L_2$ be a pair of lines in $Y$ with the intersection point $L_1\cap L_2=\{p\}$. From the exact sequence $\ses{I_C}{I_{L_{i}}}{I_{L_{i}}/I_C\cong \cO_{L_{3-i}}(-1)}$, we have an exact sequence
\begin{equation}\label{rere1}
\ses{N_{L_i/Y}}{N_{C/Y}|_{L_i}}{\CC_p}
\end{equation}
for $i=1, 2$ (\cite[Lemma A.2.1]{KPS18}). Also, from the structure sequence $\ses{\cO_{L_{i}}(-1)}{\cO_C}{\cO_{L_{3-i}}}$ and local freeness of $N_{C/Y}$, there is a short exact sequence
\begin{equation}\label{reno2}
0\lr N_{C/Y}|_{L_{i}}(-1) \lr N_{C/Y}\lr N_{C/Y}|_{L_{3-i}}\lr0.
\end{equation}

On the other hand, the general construction (so called, \emph{the Ferrand doubling}) of the (planar) double line in a smooth projective variety $Y$ was established in \cite[Section 1]{BF86}. Furthermore, the normal bundle sequence of a double line was studied in \cite[Lemma A.2.4]{KPS18}. A planar double line $D$ (i.e., $\chi(\cO_{D}(m))=2m+1$) supported on a line $L$ in $Y$ is determined by the kernel of an epimorphism
\begin{equation}\label{epi}
\phi: N_{L/Y}^*=I_L/I_L^2 \twoheadrightarrow \cO_L(-1)
\end{equation}
such that the kernel of $\phi$ is $\ker(\phi)=I_D/I_L^2$ for some ideal sheaf $I_D\subset \cO_Y$, $I_L^2\subset I_D\subset I_L$. Furthermore the kernel $I_D/I_L^2$ of the map $\phi$ in \eqref{epi} has a locally free resolution
\begin{equation}\label{lore}
\ses{\cO_L(-2)}{N_{D/Y}^*|_L}{I_D/I_L^2}
\end{equation}
where $\cO_L(-2)\cong I_L^2/I_DI_L$ and $N_{D/X}^*|_L\cong I_D/I_D^2\otimes \cO_L\cong I_D/I_DI_L$. Taking the dual functor $\cH om_L (-, \cO_L)$ into two exact sequences \eqref{epi} and \eqref{lore}, we have an exact sequence
\begin{equation}\label{dounprm}
0\lr \cO_L(1)\lr N_{L/Y}\lr N_{D/Y}|_L\lr \cO_L(2)\lr 0.
\end{equation}
On the other hand, by tensoring the normal bundle $N_{D/Y}$ into the structure sequence $\ses{\cO_L(-1)}{\cO_D}{\cO_L}$, we have
\begin{equation}\label{struse}
\ses{N_{D/Y}|_L(-1)}{N_{D/Y}}{N_{D/Y}|_L}.
\end{equation}

From \eqref{rere1} and \eqref{reno2} (resp. \eqref{dounprm} and \eqref{struse}), one can provide a sufficient condition for the smoothness of the Hilbert scheme of conics at the point corresponding to a pair of lines (resp. a double line). The proof of the following proposition can be found in \cite{KP25+} when the variety $Y$ is homogeneous. For the convenience of readers, we put a detail here.
\begin{proposition}[\protect{\cite{KP25+}}]\label{consm}
Let $Y$ be a smooth \emph{convex} variety (that is, for any morphism $f: \PP^1\lr Y$, we have $\rH^1(f^*T_Y)=0$) with a fixed embedding in a projective space $\PP^r$, then the Hilbert scheme $\bH_2(Y)$ of conics is a smooth variety.
\end{proposition}
\begin{proof}
Let $C$ be a smooth conic in $Y$. Let us choose the map $f:\PP^1\cong C$ as an embedding map. By the assumption, $\rH^1(T_Y|_C)\cong \rH^1(N_{C/Y})=0$. On the other hand, for each line $L\subset Y$, let us choose a two-to-one covering map $f: \PP^1\lr L$. Then $f_*\cO_{\PP^1}\cong\cO_L\oplus \cO_L(-1)$. Therefore, by adjunction formula and the tangent bundle sequence $\ses{T_L}{T_Y|_{L}}{N_{L/Y}}$,
\[
0=\rH^1(f^*T_Y)\cong \rH^1(T_Y\otimes f_*\cO_{\PP^1})\cong\rH^1(T_Y|_{L}\oplus T_Y|_{L}(-1))\cong\rH^1(N_{L/Y})\oplus \rH^1(N_{L/Y}(-1)).
\]
Hence one can read that the degrees of each component of the normal bundle $N_{L/Y}$ are not negative. Hence for any broken (possibly non-reduced) conic $C$ containing the line $L$, $h^1(N_{C/Y}|_{L})=h^1(N_{C/Y}|_{L}(-1))=0$ by \eqref{rere1} and \eqref{dounprm}. This implies that $h^1(N_{C/Y})=0$ by \eqref{reno2} and \eqref{struse}. Then we conclude by applying \cite[(2.12.1) and (2.14.2), Ch. I]{Kol96}.
\end{proof}
Since our variety $X$ is not convex, we can not apply Proposition \ref{consm}. But we can conclude the same result by a direct computation of the restricted bundle and a connecting map of the $\CC^*$-fixed curve (Section \ref{resec}).
\section{Results}\label{resec}
\subsection{Restricted normal bundle of fixed double lines}
We separately discuss whether the supporting line of a double line is free or not.
\begin{lemma}\label{freeob}
Let $D$ be a double line supported on a free line $D_{\text{red}}=L$ in the quintic del Pezzo fourfold $X$. Then the restricted normal bundle $N_{D/X}|_L$ fits into the exact sequence
\begin{equation}\label{fm-eq1}
0\lr \cO_L^{\oplus 2}\lr N_{D/X}|_L\lr \cO_L(2)\lr 0.
\end{equation}
\end{lemma}
\begin{proof}
Let $L$ be a free line in $X$. From the normal bundle $N_{L/X}$ in \eqref{norinx}, the exact sequence \eqref{dounprm} and $\Hom(\cO_L(1), \cO_L)=0$, we obtain the exact sequence in \eqref{fm-eq1}.
\end{proof}

Let $D$ be a double line supported on a non-free line $D_{\text{red}}=L$. From the normal bundle $N_{L/X}=\cO_L(-1)\oplus \cO_L(1)^{\oplus 2}$ in \eqref{norinx}, the local freeness of $N_{D/X}|_L$ and \eqref{dounprm}, the restricted bundle $N_{D/X}|_L$ fits into the exact sequence 
\begin{equation}\label{renormal1}
\ses{\cO_L(-1)\oplus \cO_L(1)}{N_{D/X}|_L}{\cO_L(2)}.
\end{equation}
By Serre duality,
\[
\begin{split}
\Ext_L^1(\cO_L(2),\cO_L(-1)\oplus \cO_L(1))&\cong \Ext_L^1(\cO_L(2),\cO_L(-1))\\
&\cong \rH^0(\cO_L(1))^*\cong\CC^2
\end{split}
\]
and thus the restricted normal bundle $N_{D/X}|_L$ in \eqref{renormal1} may not be a splitting type of given bundles. Surely, since the bundle $N_{D/X}|_L$ is defined on $L\cong\PP^1$, it is always a splitting one as intrinsic sense by Grothendieck theorem. In the following Lemma \ref{resno1} and Lemma \ref{resno}, we show that two types of the restricted bundle occur depending on the $\CC^*$-fixed double line in $X$.
\begin{lemma}\label{resno1}
Let $D$ be a double line (A-2) or (B-2) in Table \ref{table_2}. Then the restriction of the normal bundle of $D$ in $X$ is isomorphic to
\[N_{D/X}|_L\cong \cO_L\oplus \cO_L(1)^2.\]
That is, $N_{D/X}|_L$ is a non-splitting type extension in \eqref{renormal1}.
\end{lemma}

\begin{proof}
Since the computation of the case (B-2) is similar to that of (A-2), we only prove the claim for the latter case. Let 
\begin{equation}\label{lm-eq10}
\begin{bmatrix}1&0&a_1&a_2&a_3\\0&1&b_1&b_2&b_3\end{bmatrix}
\end{equation}
be the affine chart $U_{01}$ around the point $e_0\wedge e_1$ in $G=\Gr(2, V_5)$. In this chart, the $\CC^*$-fixed double line $D$ (resp. $X$) is defined by the ideal
\[\langle a_1^2, b_3, -a_1-b_2, a_2, a_3\rangle\; (\text{resp}. 
\langle -a_1-b_2, -a_2-a_1b_3+a_3b_1\rangle).\] 
On the other hand, let
\begin{equation}\label{lm-eq12}
\begin{bmatrix}1&c_1&0&c_2&c_3\\0&d_1&1&d_2&d_3\end{bmatrix}
\end{equation}
be the affine chart $U_{02}$ around $e_0\wedge e_2$ in $G$. In this chart, the double line $D$ (resp. $X$) is defined by the ideal
\[
\langle c_1^2, d_3, c_1-d_2, c_2, c_3\rangle\; (\text{resp}. 
\langle c_1-d_2, c_1d_2-c_2d_1+c_3\rangle).
\]
The coordinate change between $U_{01}$ and $U_{02}$ is given by row operation among two matrices in \eqref{lm-eq10} and \eqref{lm-eq12}:
\[
c_1=-\frac{a_{1}}{b_1},\; c_2=a_2-\frac{b_2a_1}{b_1},\; c_3=a_3-\frac{b_3a_1}{b_1},\; d_1=\frac{1}{b_1},\; d_2=\frac{b_2}{b_1},\; d_3=\frac{b_3}{b_1}.
\]
On the other hand, the normal bundle of $D$ in $X$ is isomorphic to $N_{D/X}\cong\cH om_{\cO_{X}}(I_{D/X}, \cO_{D})$
and thus \[N_{D/X}|_L \cong N_{D/X} \otimes \cO_L\cong\cH om_{\cO_{X}}(I_{D/X}, \cO_{L}).\] The local section space of $N_{D/X}|_L(U_{01})$ is generated by the column vectors of the identity matrix 
\begin{equation}\label{lm-eq21}
\begin{bmatrix} \alpha_1&\alpha_2&\alpha_3\end{bmatrix}=\begin{bmatrix} 1&0&0\\0&1&0\\0&0&1\end{bmatrix}
\end{equation}
with respect to the ordered basis $I_D(U_{01})=\langle a_1^2, b_3, a_3\rangle$. Let $I_D(U_{02})=\langle c_1^2, d_3, c_2\rangle$ be the generators of $I_D$ over the open set $U_{02}$. During the coordinate change from $U_{01}$ into $U_{02}$, the generators $\alpha_i$, $1\leq i\leq 3$ in \eqref{lm-eq21} have been changed into
\begin{equation}\label{lm-eq31}
\begin{bmatrix} \widetilde{\alpha}_1&\widetilde{\alpha}_2&\widetilde{\alpha}_3\end{bmatrix}=\begin{bmatrix} d_1^2&0&0\\0&d_1&0\\d_1&0&\frac{1}{d_1}\end{bmatrix}.
\end{equation}
Clearly, the second column of \eqref{lm-eq31} corresponds to the component $\cO(1)$ of the bundle $N_{D/X}|_L$ in \eqref{renormal1}. For the other generators $\alpha_1$, $\alpha_3$ as an $\cO_L$-module, the transition map $T=\begin{bmatrix}d_1^2&0\\d_1&\frac{1}{d_1}\end{bmatrix}$ by deleting the second row and column of the transition map in \eqref{lm-eq31} can be regarded as a morphism from a free module generated by $\alpha_1$, $\alpha_3$ as an $\cO_{U_{01}}(U_{01}\cap U_{02})$-module to a free module as an $\cO_{U_{02}}(U_{01}\cap U_{02})$-module. Being careful to the domain and range of the transition map $T$, we know that the equivalent one of the transition map $T$ is given by the following form
\[
\widetilde{T}=UTV,
\]
where trivialization maps are of the forms $U=\begin{bmatrix}1&u_1\\u_2&1\end{bmatrix}$ (resp. $V=\begin{bmatrix}1&v_1\\v_2&1\end{bmatrix}$), $u_1\cdot u_2=0$, $u_i\in \CC[d_1]$ (resp. $v_1\cdot v_2=0$, $v_i\in \CC[b_1]=\CC[\frac{1}{d_1}]$). For suitable choices of low (or upper) triangular matrixs $U$ and $V$, one can easily see that the simplified form of the given transition map $T$ is $\widetilde{T}=\begin{bmatrix}0&-1\\d_1&0\end{bmatrix}$. That is, it corresponds to the two component $\cO_L(1)\oplus \cO_L$ of the bundle $N_{D/X}|_L$. After all, we proved the claim.
\end{proof}

\begin{remark}
Since $-1=\deg\{T_{22}=\frac{1}{d_1}\}<\deg\{T_{21}=d_1\}<\deg \{T_{11}=d_1^2\}=2$ for the transition map $T=[T_{ij}]$ in the proof of Lemma \ref{resno1}, the non-splitness of the bundle $N_{D/X}|_L$ follows from a general construction of an extension bundle of given two line bundles (cf. \cite[Lemma 2.1]{CCK05}).
\end{remark}

\begin{lemma}\label{resno}
Let $D$ be a double line in Table \ref{table_2} except the two cases (A-2) and (B-2). Then the restriction on $L$ of the normal bundle $N_{L/X}$ of $D$ in $X$ is isomorphic to
\[N_{D/X}|_L\cong  \cO_L(-1)\oplus \cO_L(1)\oplus \cO_L(2).\]
That is, $N_{D/X}|_L$ is a splitting extension in \eqref{renormal1}.
\end{lemma}

\begin{proof}
Since the other case computation is similar to (C-1), we only write down the detail for this case. Let
\begin{equation}\label{lm-eq0}
\begin{bmatrix}1&0&a_2&a_3&a_4\\0&1&b_2&b_3&b_4\end{bmatrix}
\end{equation}
be the affine chart $U_{01}$ around the point $e_0\wedge e_1$ in $G=\Gr(2, V_5)$. In this chart, the $\CC^*$-fixed double line $D$ (resp. $X$) is defined by the ideal
\[\langle b_2^2, b_3, a_2, a_3, a_4\rangle\; (\text{resp}. 
\langle a_2+b_3, -a_3-a_2b_4+a_4b_2\rangle).\] 
On the other hand, let
\begin{equation}\label{lm-eq1}
\begin{bmatrix}1&f_1&f_2&f_3&0\\0&g_1&g_2&g_3&1\end{bmatrix}
\end{equation}
be the affine chart $U_{04}$ around $e_0\wedge e_4$ in $G$. In this chart, the double line $D$ (resp. $X$) is defined by the ideal
\[
\langle g_2^2, g_3, f_1, f_2, f_3\rangle\; (\text{resp}. 
\langle f_1g_2-f_2g_1-g_3, f_1g_3-f_3g_1-f_2\rangle).
\]
The coordinate change between $U_{01}$ and $U_{04}$ is given by row operation among two matrices in \eqref{lm-eq0} and \eqref{lm-eq1}:
\[
f_1=-\frac{a_{4}}{b_4},\; f_2=a_2-\frac{a_4b_2}{b_4},\; f_3=a_3-\frac{a_4b_3}{b_4},\; g_1=\frac{1}{b_4},\; g_2=\frac{b_2}{b_4},\; g_3=\frac{b_3}{b_4}.
\]
The local section space of $N_{D/X}|_L(U_{01})$ is generated by the column vectors of the identity matrix 
\begin{equation}\label{lm-eq2}
\begin{bmatrix} \alpha_1&\alpha_2&\alpha_3\end{bmatrix}=\begin{bmatrix} 1&0&0\\0&1&0\\0&0&1\end{bmatrix}
\end{equation}
with respect to the ordered basis $I_D(U_{01})=\langle b_2^2, b_3, a_4\rangle$. Let $I_D(U_{04})=\langle g_2^2, f_1, f_3\rangle$ be the generators of $I_D$ over the open set $U_{04}$. During the coordinate change from $U_{01}$ into $U_{04}$, the generators $\alpha_i$, $1\leq i\leq 3$ in \eqref{lm-eq2} have been changed into three column of the matrix
\begin{equation}\label{lm-eq3}
\begin{bmatrix} \widetilde{\alpha}_1&\widetilde{\alpha}_2&\widetilde{\alpha}_3\end{bmatrix}=\begin{bmatrix} g_1^2&0&0\\0&0&-g_1\\0&\frac{1}{g_1}&0\end{bmatrix}.
\end{equation}
By reading the degree of the variable $g_1$ of the columns of \eqref{lm-eq3}, we finish the proof of the claim.
\end{proof}

\subsection{The restricted normal bundle of fixed reducible conics}
To prove the smoothness of Hilbert scheme $\bH_2(X)$, we consider the case of a pair of non-free lines.
\begin{lemma}\label{lm-pair}
For the reduced $\CC^*$-fixed conic $C=L\cup L'$ where both of the two lines $L$ and $L'$ are non-free, the restricted normal bundle of $C$ in $X$ is
\[N_{C/X}|_{L}\cong \cO_{L}(-1)\oplus\cO_{L}(1)\oplus \cO_{L}(2).\]
\end{lemma}
\begin{proof}
Clearly, The lines $L$ and $L'$ are $\CC^*$-fixed lines. Thus from the classification of $\CC^*$-fixed, non-free line, we need to consider the three cases: (a)$\cup$(c), (c)$\cup$(d) and (b)$\cup$(d) in Table \ref{table_1}. Since the computation is similar to each other, we present the proof for the first case:
\[
C=L\cup L'=\{e_0\wedge(ue_1+ ve_2)\}\cup \{e_0\wedge(ue_1+ ve_4)\},\; [u:v]\in \PP^1.
\]
Let 
\begin{equation}\label{qm-eq0}
\begin{bmatrix}1&0&a_2&a_3&a_4\\0&1&b_2&b_3&b_4\end{bmatrix}
\end{equation}
be the affine chart $U_{01}$ around the point $e_0\wedge e_1$ in $G=\Gr(2, V_5)$. In this chart, the $\CC^*$-fixed double line $D$ (resp. $X$) is defined by the ideal
\[\langle b_2b_4, b_3, a_2, a_3, a_4\rangle\; (\text{resp}. 
\langle a_2+b_3, -a_3-a_2b_4+a_4b_2\rangle).\] 

On the other hand, let
\begin{equation}\label{qm-eq1}
\begin{bmatrix}1&c_1&0&c_2&c_3\\0&d_1&1&d_2&d_3\end{bmatrix}
\end{equation}
be the affine chart $U_{02}$ around $e_0\wedge e_2$ in $G$. In this chart, the reducible conic $D$ (resp. $X$) is defined by the ideal
\[
\langle c_1, c_2, c_3, d_2,d_3\rangle\; (\text{resp}. 
\langle c_1-d_2, c_1d_2-c_2d_1+c_3\rangle).
\]
Since the line $L$ is supported on $U_{01}\cup U_{02}$, we need to compute the transition map of the restricted bundle $N_{C/X}|_{L}$ between $U_{01}$ and $U_{02}$.
The coordinate change between $U_{01}$ and $U_{02}$ is given by row operation among two matrices in \eqref{qm-eq0} and \eqref{qm-eq1}:
\[
c_1=-\frac{a_2}{b_2},\;c_2=a_3-\frac{b_3}{b_2}a_2,\;c_3=a_4-\frac{b_4}{b_2}a_2,\;d_1=\frac{1}{b_2},\;d_2=\frac{b_3}{b_2},\;d_3=\frac{b_4}{b_2}
\]
The local section space of $N_{C/X}|_L(U_{01})$ is generated by the column vectors of the identity matrix 
\begin{equation}\label{qm-eq21}
\begin{bmatrix} \alpha_1&\alpha_2&\alpha_3\end{bmatrix}=\begin{bmatrix} 1&0&0\\0&1&0\\0&0&1\end{bmatrix}
\end{equation}
with respect to the ordered basis $I_C(U_{01})=\langle b_2b_4, b_3, a_4\rangle$. Let $I_C(U_{02})=\langle d_3, d_2, c_2\rangle$ be the generators of $I_C$ over the open set $U_{02}$. During the coordinate change from $U_{01}$ into $U_{02}$, the generators $\alpha_i$, $1\leq i\leq 3$ in \eqref{qm-eq21} have been changed into
\begin{equation}\label{qm-eq31}
\begin{bmatrix} \widetilde{\alpha}_1&\widetilde{\alpha}_2&\widetilde{\alpha}_3\end{bmatrix}=\begin{bmatrix} d_1^2&0&0\\0&d_1&0\\0&0&\frac{1}{d_1}\end{bmatrix}.
\end{equation}
By reading the degree of the variable $d_1$ of the columns of \eqref{qm-eq31}, we finish the proof of the claim.
\end{proof}

\subsection{Fixed smooth conics and its normal bundle}
Since the morphism $\widetilde{\Phi}$ in Proposition \ref{bidiagconic} is a $\CC^*$-equivariant one, we will find the irreducible $\CC^*$-fixed conics contained in (possibly reducible) quadric surfaces.
\begin{proposition}
There does not exist a $\CC^*$-fixed smooth conic in any $\sigma_{3,1}$-plane $P\subset X$.
\end{proposition}
\begin{proof}
Let $C$ be such a conic in the plane $P$. Then $P$ must be fixed by the torus action. Since $P$ is parameterized by the vertex conic $\PP^1\cong C_{v}$ in Proposition \ref{planespace}, the possible cases are $P=\PP(V_1\wedge V_4)$ for the pair $(V_1, V_4)=(\langle e_0\rangle , \langle e_0, e_1, e_2, e_4\rangle)$ or $(\langle e_4\rangle , \langle e_0, e_1, e_3, e_4\rangle)$. The latter case is symmetric to first case with respect to the weight and thus it is enough to consider the first case. In this case, the Hilbert scheme of conics is parameterized by $\bH_2(P)=\PP(\Sym^2(V_1\wedge V_4)^*)$. The weights in the symmetric power $\Sym^2(V_1\wedge V_4)^*$ are different from each other. Therefore, each conic fixed by $\CC^*$-action should be a pair of lines or double lines. So we finished the proof.
\end{proof}
The below Proposition \ref{fixcon22} and Proposition \ref{fixcon23} can be proved by using a geometry of  quintic del Pezzo variety $X$, which is similar method to the proof of \cite[Theorem 3.2]{IM09}. But, in this paper, we directly compute the normal bundle of smooth conics.
\begin{proposition}\label{fixcon22}
There exists one parameter family of smooth conics in the unique $\sigma_{2,2}$-plane $S=\PP(e_0\wedge e_1, e_0\wedge e_4, e_1\wedge e_4)$ as follows.
\begin{equation}\label{psmcon}
(se_0+te_1)\wedge (se_1+ate_4), \; [s:t]\in \PP^1, a\in \CC^*.
\end{equation}
Furthermore, the normal bundle of for any conic $C_a$ in $X$ parameterized by $a\in \CC^*$ is isomorphic to \[N_{C_a/X}\cong \cO_{\PP^1}(4)\oplus \cO_{\PP^1}\oplus \cO_{\PP^1}.\]
\end{proposition}
\begin{proof}
By considering the weights of the plane $S$, one can easily see the torus fixed conics in \eqref{psmcon}. Let us regard the fixed curve in \eqref{psmcon} as a regular map parameterized by $[s:t]\in\PP^1$. The tangent bundle of $G=\Gr(2, V_5)$ is isomorphic to $T_G\cong \cU^*\otimes \cQ$ for the universal sub (resp. quotient)-bundle $\cU$ (resp. $\cQ$) of Grassmannian variety $G$. The restricted tangent bundle of $G$ is isomorphic to $T_G|_{C_a}\cong \cO_{\PP^1}(3)^2\oplus \cO_{\PP^1}(1)^4$ because $\cQ|_{C_a}= \cO_{\PP^1}(2)\oplus \cO_{\PP^1}^2$. By plugging in the tangent bundles sequence $\ses{T_{C_a}\cong \cO_{\PP^1}(2)}{T_G|_{C_a}}{N_{C/G}}$, we have
\begin{equation}\label{comp5}
N_{C_a/G}\cong \cO_{\PP^1}(4)\oplus \cO_{\PP^1}(1)^4.
\end{equation}
Let 
\[
0\lr N_{C_a/X}\lr N_{C_a/G}\stackrel{\nabla H_1\oplus \nabla H_{-1}} {\lr} N_{X/G}|_{C_a}= \cO_{\PP^1}(2)\oplus \cO_{\PP^1}(2)\lr 0
\]
be the normal bundle sequence of the nested inclusion $C_a\subset X\subset G$ where the map $\nabla H_1\oplus \nabla H_{-1}:=H$ arises from the composition map $T_G|_{C_a}\twoheadrightarrow N_{C_a/G}\twoheadrightarrow N_{X/G}|_{C_a}$. Let us compute the kernel of the map $H$ by a local computation. The local chart of the tangent space $T_G$ around the curve $C_a$: $(e_0+te_1)\wedge (e_1+a te_4)$ is given by $(e_0+a^{-1}v_1e_1+v_2e_2+v_3e_3+v_4e_4)\wedge (e_1+w_2e_2+w_3e_3+v_1e_4)$ where the local chart of the normal space $N_{C_a/G}$ is $(v_4, v_2, v_3, w_2, w_3)$. By using the defining equations of the hyperplane in \eqref{defhy}, one can easily see that the map $H$ is presented by
\[
H=\begin{bmatrix}
0&-1&0&a^{-1}v_1&-1\\
w_2&-v_1&-1&v_4&a^{-1}v_1
\end{bmatrix}.
\]
When we restrict the map $H$ to the curve $C_a$, one can read that the component $\cO_{\PP^1}(1)^4$ (resp. $\cO_{\PP^1}(4)$) of the normal bundle $N_{C_a/G}$ in \eqref{comp5} surjectively maps to $N_{X/G}|_{C_{a}}$ (resp. zero). Hence the kernel $N_{C_a/X}$ of the map $H$ is isomorphic to the one in the statement.
\end{proof}

\begin{proposition}\label{fixcon23}
A $\CC^*$-fixed smooth conic $C$ contained in a quadric surface (cone) $\Gr(2, V_4) \cap H_1\cap H_{-1}\subset X$ is given in Table \ref{table_3}.
Furthermore for each $C$, the normal bundle of $C\cong \PP^1$ in $X$ is isomorphic to
\[N_{C/X}\cong \cO_{\PP^1}(2)\oplus \cO_{\PP^1}(1)\oplus \cO_{\PP^1}(1).\]
\end{proposition}
\begin{table}[h]
   \setlength{\tabcolsep}{4pt}
   \centering
   \begin{tabular}{|c|c|c|} 
    \hline
    \rule{0pt}{2.6ex} 
    \emph{(i)} & \emph{(ii)} & \emph{(iii)} \\ [0.5ex] 
    \hline
    &  &   \\ [-1.5ex]
    $(se_1+te_2)\wedge(te_3+se_4)$ & $(se_1+te_3)\wedge (-se_0+te_2)$ & $(se_0+te_3)\wedge(bse_2+te_4)$ \\[1.5ex]
    \hline
 \end{tabular}  
   \caption{$\mathbb{C}^{\ast}$-fixed smooth conics parameterized by $[s:t]\in \PP^1$}
   \label{table_3}
 \end{table} 
\begin{proof}
Let $C$ be a $\CC^*$-fixed smooth conic in $X\subset G$. Then the linear spanning $\PP^3=\PP(V_4)$ of $C$ in $\PP(V_5)$ is also fixed under the $\CC^*$-action. Since the weights of $V_5$ are pairwise different, each coordinate hyperplane $V_4\subset V_5$ is a $\CC^*$-fixed one. We find the defining equations of the torus fixed linear space $\PP^3=\PP(\wedge^2 V_4)\cap H_1\cap H_2$ in $\PP^9$ and its weights. Note that $\langle \Gr(2, V_4)\rangle=\PP(\wedge^2 V_4)$. For instance, let $V_4=V_5/\langle e_0\rangle$. Let us collect $\CC^*$-fixed conic in the quadric surface defined by the intersection of these $\PP^3$'s and $\Gr(2,V_4)$. The coordinates of the space $\PP(\wedge^2 V_4)$ is $p_{ij}$, $1\leq i <j \leq 4$ and thus the linear space $\PP^3$ is defined by $p_{12}=p_{13}-p_{24}=p_{0k}=0$, $1\leq k\leq 4$. Hence the weights are $-1$, $-2$, $0$, $-3$. For each coordinate plane in this $\PP^3$, one can obtain the conic. That is, for the plane defined by $p_{34}=0$, then $p_{12}=p_{13}-p_{24}=p_{12}p_{34}-p_{24}p_{13}+p_{23}p_{14}=0$ is a smooth $\CC^*$-fixed conic which is parametrically defined in (i) of Table \ref{table_1}. By the reversing the weights, that is, when $V_4=V_5/\langle e_4\rangle$, we obtain the smooth conic (ii) in Table \ref{table_3}. On the other hand, one can easily check that when $V_4=V_5/\langle e_2\rangle$ or $V_5/\langle e_3\rangle$, there does not exist any smooth conic. Lastly, if $V_4=V_5/\langle e_1\rangle$, then we obtain the smooth conics (iii) in Table \ref{table_3} parameterized by $b\in \CC^*$. On the other hand, by a similar computation in the proof of Proposition \ref{fixcon22}, we obtain the normal bundle in the claim.
\end{proof}

\subsection{Proof of Theorem \ref{mainthm}}
We firstly prove the smoothness of $\bH_2(X)$ and then show that there is an affine subspace of dimension $7$ in $\bH_2(X)$ with a help of Theorem \ref{prop:onetoone}.
\begin{proof}[Proof of the smoothness of $\bH_2(X)$]
We show that for each $\CC^*$-fixed conic $C$ in $X$, the obstruction space vanishes $\rH^1(N_{C/X})=0$ (Lemma \ref{fsws} and \cite[(2.12.1) and (2.14.2), Ch. I]{Kol96}).

 \noindent\textbf{Non-reduced case}:
Let $C$ be a non-reduced conic supported on a free line $C_{\text{red}}=L$. Then from Lemma \ref{freeob}, $\rH^1(N_{D/X}|_L(i))=0$ for $i=0$ or $-1$ and thus $\rH^1(N_{D/X})=0$ by \eqref{struse}. On the other hand, let us assume that $C$ is a double line supported on a non-free line $C_{\text{red}}=L$. For the case (A-2) and (B-2) in Table \ref{table_2}, we have $h^1(N_{C/X})=0$ by Lemma \ref{resno1} and \eqref{struse}. For the remaining cases, since the computation is similar to each other, we only prove the claim for the case (C-1) in Table \ref{table_2}. From the twisted structure sequence \eqref{struse} and Lemma \ref{resno}, we have a long exact sequence
\begin{equation}\label{connect}
\cdots\lr \rH^0(N_{C/X}|_L)\stackrel{\partial}{\lr} \rH^1(N_{C/X}|_L(-1))\lr \rH^1(N_{C/X}) \lr \rH^1(N_{C/X}|_L)=0.
\end{equation}
We claim that the connecting homomorphism $\partial$ in \eqref{connect} is a non-zero map and thus $\rH^1(N_{C/X})=0$. Let us use the same notation as in Lemma \ref{resno}. Let $C^{i}(\cF)$ be the $i$-th \v{C}ech complex of a coherent sheaf $\cF\in \text{Coh}(X)$ with respect to the affine open covering $\{U_{01}, U_{04}\}$. Note that the double line $C$ (and line $L$) is supported only on these two open charts. Since $h^1(N_{C/X}|_L(-1))=1$ and $h^1(N_{C/X}|_L) =0$ by Lemma \ref{resno}, we have a commutative diagram:
\[
\xymatrix{&&\rH^0(N_{C/X})\ar@{^{(}->}[d]\ar[r]&\rH^0(N_{C/X}|_L)\ar@{^{(}->}[d]\ar@/^3.8pc/[llddd]^{\partial}&\\
0\ar[r]&C^0(N_{C/X}|_L(-1))\ar[r]\ar[d]&C^0(N_{C/X})\ar[r]\ar[d]&C^0(N_{C/X}|_L)\ar[r]\ar[d]&0\\
0\ar[r]&C^1(N_{C/X}|_L(-1))\ar[r]\ar@{->>}[d]&C^1(N_{C/X})\ar[r]\ar@{->>}[d]&C^1(N_{C/X}|_L)\ar[r]\ar[d]&0\\
&\CC\cong \rH^1(N_{C/X}|_L(-1))\ar[r]&\rH^1(N_{C/X})\ar[r]&\rH^1(N_{C/X}|_L)=0.&
}\]
Note that $N_{C/X}|_L(-1)\cong\cH om_{\cO_X}(I_C, I_{L}/I_{C})$ (cf. \eqref{epi}). The generator of the cohomology group $\rH^1(N_{C/X}|_L(-1))$ as an element of the \v{C}ech complex is presented by
\[
\gamma\begin{bmatrix}g_2^2\\f_1\\f_3\end{bmatrix}=\begin{bmatrix}0\\0\\ \frac{g_2}{g_1}\end{bmatrix}
\]
on the intersection $U_{01}\cap U_{04}$. If we choose $(\alpha, \beta)\in C^0(N_{C/X})$ such that
\[
\alpha\begin{bmatrix}b_2^2\\b_3\\a_4\end{bmatrix}=\begin{bmatrix}0\\-b_2\\ b_4\end{bmatrix}, \;
\beta\begin{bmatrix}g_2^2\\f_1\\f_3\end{bmatrix}=\begin{bmatrix}0\\-1\\  -\frac{g_2}{g_1}\end{bmatrix},
\]
then $\alpha-\beta=\gamma\in C^1(N_{C/X})$. Also for the restrictions $\alpha|_L$ and $\beta|_{L}$, the difference is $\alpha|_L-\beta|_{L}=0\in C^1(N_{C/X}|_L)$ and thus $[(\alpha|_{L}, \beta|_{L})]\in \rH^0(N_{C/X}|_L)$. That is, \[\partial([(\alpha|_{L}, \beta|_{L})])=\gamma\] in \eqref{connect} and thus the obstruction space vanishes $\rH^1(N_{C/X})=0$.

 \noindent\textbf{Reducible case}:
Let $C=L\cup L'$ be a $\CC^*$-fixed, a pair of lines. Note that if a line $L$ is free one, then the restricted normal bundle is of the form $N_{C/X}|_{L}\cong \cO_{L}(1)^2\oplus \cO_{L}$ or $\cO_{L}(2)\oplus \cO_{L}(1)^2$ by \eqref{rere1}. In any case, $h^1( N_{C/X})=0$ because the smallest degree term of the restricted normal bundle $N_{C/X}|_{L}(-1)$ on $L$ is $\geq -1$. So let us consider the case where $L$ and $L'$ are non-free. Since $h^1(N_{C/X}|_L(-1))=h^1(N_{C/X}|_{L'}(-1))=1$ and $h^1(N_{C/X}|_L)=h^1(N_{C/X}|_{L'})=0$ by Lemma \ref{lm-pair}, the method is completely the same as the proof of non-reduced case and thus we skip the proof.

 \noindent\textbf{Smooth case}: Clearly, from Proposition \ref{fixcon22} and Proposition \ref{fixcon23}, we have $h^1(N_{C/X})=0$.
\end{proof}

Since the Hilbert scheme $\bH_2(X)$ is birational to the Grassmannian bundle $\bS(X)$ (Proposition \ref{bidiagconic}), the following result is obvious one in the sense of birational equivalence. In this paper, we follow the localization technique developed in \cite[Section 4]{MNOP06}.
\begin{corollary}
The Hilbert scheme $\bH_2(X)$ is a rational variety.
\end{corollary}
\begin{proof}
Let $C$ be the non-reduced conic of the case (A-2) in Table \ref{table_2}. By the proof of Theorem \ref{mainthm}, we know that $\rH^1(N_{C/X})=0$. Hence the weight representation of the tangent space of Hilbert scheme $\bH_2(X)$ at $[C]$ can be obtained from the \v{C}ech complex of $N_{C/X}$. Under the notation of Lemma \ref{resno1}, the affine space $U_{01}\cong \CC^4$ (resp. $U_{02}$)   has coordinates $\{a_1, b_1, a_3, b_3\}$ (resp. $\{c_1, c_2, d_1, d_3\}$) with weights $-1$, $1$, $-4$, $-2$ (resp. $-2$, $-3$, $-1$, $-3$) respectively, which are inherited from the original weights of the vector space $V_5$. Applying for these weights to the local section spaces $N_{C/X}(U_{0i})=\cH om (I_C(U_{0i}),\cO_C(U_{0i}))$ for $i=1, 2$ and its intersection part $N_{C/X}(U_{01}\cap U_{02})$, we have
\[\begin{split}
&\rH^0(N_{C/X})=\chi(N_{C/X})=C^0(N_{C/X})-C^1(N_{C/X})\\
&=(t^2+t^2+t^4)(\frac{1}{1-t}+\frac{t^{-1}}{1-t})+(t^4+t^3+t^3)(\frac{1}{1-t^{-1}}+\frac{t^{-2}}{1-t^{-1}})\\
&-\{(t^2+t^2+t^4)(\frac{1}{1-t}+\frac{t^{-1}}{1-t})+(t^2+t^2+t^4)(\frac{t^{-1}}{1-t^{-1}}+\frac{t^{-1}}{1-t^{-1}}\cdot t^{-1})\}\\
&=t^4+2t^3+2t^2+2t.
\end{split}\]
Since all of signs of the weight of $\rH^0(N_{C/X})$ are positive, we proved the claim by Theorem \ref{prop:onetoone}.
\end{proof}

\end{document}